\documentclass[oneside,reqno]{amsart}
\usepackage{amssymb,amsmath}
\usepackage[english]{babel}
\usepackage{amsfonts}
\usepackage{color}
\usepackage{amsthm}
\usepackage[colorlinks=true,linkcolor=NavyBlue, citecolor=NavyBlue,urlcolor=NavyBlue]{hyperref}
\usepackage{graphicx}
\usepackage{mathtools}
\usepackage[usenames,dvipsnames]{pstricks}
\usepackage{bm}
\usepackage{amsmath}
\usepackage{listings}
\usepackage{hyperref}
\usepackage[usenames,dvipsnames]{pstricks}
 \usepackage{epsfig}
 \usepackage{pst-grad} 
 \usepackage{pst-plot} 

\usepackage{xy}
\usepackage[draft]{fixme}
\newcommand\deq{\mathrel{\stackrel{\makebox[0pt]{\mbox{\normalfont\tiny def}}}{=}}}

\let\phi\varphi
\theoremstyle{plain}
\newtheorem{theorem}{Theorem}[section]
\newtheorem{lemma}[theorem]{Lemma}
\newtheorem{proposition}[theorem]{Proposition}
\newtheorem{corollary}[theorem]{Corollary}

\theoremstyle{remark}
\newtheorem{remark}{Remark}

\theoremstyle{definition}
\newtheorem{definition}[theorem]{Definition}
\newtheorem{example}[theorem]{Example}
\newtheorem*{notation*}{Notation}

\def\hl#1{{\color{NavyBlue}#1}}
\usepackage{listings}
\lstdefinelanguage{GAP}{%
 morekeywords={%
 Assert,Info,IsBound,QUIT,%
 TryNextMethod,Unbind,and,break,%
 continue,do,elif,%
 else,end,false,fi,for,%
 function,if,in,local,%
 mod,not,od,or,%
 quit,rec,repeat,return,%
 then,true,until,while%
 },%
 sensitive,%
 morecomment=[l]\#,%
 morestring=[b]",%
 morestring=[b]',%
}[keywords,comments,strings]

\usepackage[T1]{fontenc}
\usepackage[variablett]{lmodern}
\usepackage{xcolor}
\usepackage[foot]{amsaddr}
\usepackage{enumerate}
\usepackage{arydshln}
\let \l \lambda
\let \hat \widehat
\newcommand{\listP}{\l = (\l_1,\l_2, \dots, \l_t)}
\newcommand{\missP}{\mu_1 < \mu_2 < \dots < \mu_m}
\newcommand{\Ml}{\mathcal M_\l}

\newcommand{\pt}{\widetilde{\pi_n}}
\newcommand{\al}{\alpha}
\newcommand{\mup}[1]{\widetilde{\mathbb{U}}_{#1}}
\newcommand{\up}{\mathbb{U}}
\newcommand{\bl}{\bullet}
\newcommand{\mex}[1]{{\mathrm{mex}}(#1)}

\newcommand{\dist}{\mathbb D}
\newcommand{\mA}{\mathcal{A}}
\newcommand{\mB}{\mathcal{B}}
\newcommand{\mC}{\mathcal{C}}
\newcommand{\mD}{\mathcal{D}}
\newcommand{\mdA}[1]{\mathcal{A}^*_{#1}}
\newcommand{\mdB}[1]{\mathcal{B}^*_{#1}}
\newcommand{\mdC}[1]{\mathcal{C}^*_{#1}}
\newcommand{\mdD}[1]{\mathcal{D}^*_{#1}}

\begin{document}
\title[On the maximal part in unrefinable partitions of triangular numbers]{On the maximal part in unrefinable partitions of triangular numbers} 
 \author{Riccardo Aragona$^1$, Lorenzo Campioni$^1$, Roberto Civino$^1$}
 
 \address{$^1$
 DISIM - Universit\`a degli Studi dell'Aquila, Italy}    
 \author{Massimo Lauria$^2$}
 \address{$^2$ Dipartimento di Scienze Statistiche - Sapienza, Università di Roma, Italy}    

\email{riccardo.aragona@univaq.it, roberto.civino@univaq.it, lorenzo.campioni1@graduate.univaq.it}
\email{ massimo.lauria@uniroma1.it}
\date{} \thanks{R. Aragona and R. Civino are members of INdAM-GNSAGA
 (Italy). R. Civino is partially funded by the Centre of excellence
 ExEMERGE at the University of L'Aquila.}


\begin{abstract}
A partition into distinct parts is refinable if one of its parts $a$ can be replaced by
two different integers which do not belong to the partition and whose sum is $a$, and it is unrefinable otherwise. 
Clearly, 
the condition of being unrefinable imposes on the partition a non-trivial limitation on the size of the largest part and on the possible distributions of the parts.
We prove a $O(n^{1/2})$-upper bound for the largest part in an unrefinable partition of $n$, and we call maximal those which reach
the bound. We show a complete classification of maximal unrefinable partitions for triangular numbers, proving that if $n$ is even there exists only
one maximal unrefinable partition of $n(n+1)/2$, and that if $n$ is odd the number of such partitions equals the number of partitions
of $\lceil n/2\rceil$ into distinct parts. In the second case, an explicit bijection is provided.
\end{abstract}

\maketitle


\section{Introduction}
Integer partitions into distinct parts may appear in several areas of mathematics, sometimes unexpectedly. For example, 
they have been recently shown to be linked to the set of generators of groups in  a group-theoretical problem related 
to cryptography~\cite{Aragona2019,aragona2021rigid}. In particular, Aragona et al.\ showed that the generators of a given group are linked to partitions into distinct parts
which satisfy a condition of \emph{non-refinability}~\cite{aragona2021unrefinable} together with a condition on the minimal excludant.
This motivates us to investigate some combinatorial aspects of \emph{unrefinable partitions}, i.e.\ those in which no
part can be written as the sum of 
two different integers which do not belong to the partition, which to our 
knowledge have not been investigated so far (cf.~the On-Line Encyclopedia of Integer Sequences for the first values~\cite[\url{https://oeis.org/A179009}]{OEIS}).

Computational results suggest that the maximal part in an unrefinable partition of $n$ is approximatively $\sqrt{n}$. 
In this paper, we first prove a matching upper bound for the maximal part and then we define \emph{maximal unrefinable partitions} as those which reach the bound.
As a main contribution, we provide a complete classification of maximal unrefinable partitions for triangular numbers.
We constructively prove, denoting by $T_n$ the $n$-th triangular number, that for even $n$ there exists exactly one maximal unrefinable partition of $T_n$.
For  odd $n$, we obtain a lower bound for the minimal excludant for the maximal unrefinable partitions of $T_n$, defined to be the least integers which is not a part~\cite{fraenkel2015harnessing} and  which has been  investigated also recently by other authors~\cite{andrews2019,Hopkins2022}.
The knowledge of a bound on the minimal excludant, among other considerations, allows us to show an explicit bijection between the set of the maximal unrefinable partitions of $T_n$ and the set of partitions of $\lceil n/2\rceil$ into distinct parts in the classical sense~\cite{andrews1998theory}.\\

The remainder of the paper is organized as follows: in Sec.~\ref{sec:prel} we introduce the notation and define unrefinable partitions.
In Sec.~\ref{sec:bounds} we prove two upper bounds for the maximal part in an unrefinable partition of $n$, distinguishing the case 
when $n$ is a triangular number and when it is not. The classification theorem, i.e.\ Theorem~\ref{th:maxevenodd}, is proved in Sec.~\ref{sec:maximal}, which also contains the
result on triangular numbers of an even number. The odd case is  developed in Sec~\ref{sec:odd}, which concludes the paper. In particular, 
we show in Theorem~\ref{magari} a bijective proof that the number of maximal unrefinable partitions of $T_n$  equals the number of partitions
of $\lceil n/2\rceil$ into distinct parts. 


\section{Preliminaries}\label{sec:prel}
Let $N \in \mathbb N$. A \emph{partition} of $N$ is a finite sequence $\l = (\l_1,\l_2, \dots, \l_t)$ of positive integers such that $\l_1\leq \l_2\leq  \dots \leq \l_t$ and $\sum_{i=1}^{t}\l_i= N$. 
When $\l$ is a partition of $N$ we write $\l \vdash N$. 
Each $\l_i$ is called a \emph{part} of the partition $\l$ and we call $\l_t$ its \emph{maximal part}. 
We denote by $(\l_1,\l_2, \dots, \l_{i-1},\hat{\l_i}, \l_{i+1}, \dots, \l_t)$ the partition 
$(\l_1,\l_2, \dots, \l_{i-1}, \l_{i+1}, \dots, \l_t)$ where the part $\l_i$ is removed.

The partition $\listP$ is a \emph{partition into distinct parts} if $\l_1 < \l_2 < \dots < \l_t$ and $t \geq 2$, i.e.\ if each part appears exactly once. The set $\dist_N$ denotes the set of all the partitions of $N$ into distinct parts. 
If $\listP \in \dist_N$, the integers belonging to \[\mathcal M_\l \deq \{1,2,\dots, {\l_t}\} \setminus \{\l_1,\l_2,\dots, \l_t\}\] are called the \emph{missing parts} of $\lambda$, and are denoted by $\mu_1 < \mu_2 < \dots <\mu_m$, for some $m \geq 0$.  The least integer which is not a part of $\l$, i.e. $\mu_1$, is the \emph{minimal excludant} of $\l$~\cite{fraenkel2015harnessing}. We denote this by writing $\mu_1  = \mex{\l}$, taking $\mex{\l} = 0$ when $\mathcal M_\l = \emptyset$ as it is customary in the literature.

\begin{definition}\label{def_unref}
Let $N \in \mathbb N$. Let $\listP$ be a partition of $N$ into distinct parts and let $\mu_1 < \mu_2 < \dots <\mu_m$ be its missing parts.  The partition $\l$ is \emph{refinable} if there exist $1 \leq \ell \leq t$ and $1\leq i<j \leq m $  such that $\mu_i+\mu_j = \l_{\ell}$, and \emph{unrefinable} otherwise. The set of unrefinable partitions is denoted by $\up$, and by $\up_N$ we denote those whose sum of the parts is $N$.
\end{definition}

\begin{definition}
Let $n \in \mathbb N$. We denote by $T_n$ the \emph{$n$-th triangular number}, i.e. 
\[
T_n \deq \sum_{i=1}^n i = \dfrac{n(n+1)}{2}.
\]
 The \emph{complete partition} $\pi_n \deq (1,2,\ldots,n)$ is the partition of $T_n$ with no missing parts. 
\end{definition}
Notice that every complete partition is unrefinable. The same holds, by definition, for partitions with a single missing part. In particular, if $N=T_n$ for some $n$, then $\pi_n$ is an unrefinable partition of $N$. Otherwise, if $n$ is the least integer such that $N < T(n)$, then 
\begin{equation}\label{eq:pind}
\pi_{n,d} \deq (1,2,\dots,d-1,\hat{d},d+1,\dots, n)
\end{equation}
is an unrefinable partition of $N$, where $d =T_n - N$. 
In general, the admissible number of missing parts in an unrefinable partition is bounded as in the following result.

\begin{lemma}\label{lemma:bound}
Let $\listP$ be unrefinable and let $\missP$ be the missing parts. Then the number of missing parts $m$ is bounded by
\begin{equation}\label{bound1}
m \leq \left\lfloor \dfrac{\l_t}{2}\right\rfloor.
\end{equation}
\end{lemma}
\begin{proof}
Let us start by observing that $\l_t-\mu_i \in \l$ for $1 \leq i \leq m$, otherwise from $\l_t-\mu_i, \mu_i \in \Ml$ we obtain $ (\l_t-\mu_i)+\mu_i=\l_t\in \l$ and thus $\l$ is refinable.
We prove the claim considering the complete partition $\pi_{\l_t}$ and removing from this the maximum number of parts different from $\l_t$. 
For the previous observation, each candidate part $\mu_i$ to be removed  has a counterpart $\l_t-\mu_i$ in the partition.
The bound of Eq.~\eqref{bound1} depends on the fact that this process can be repeated no more than $\lfloor {\l_t}/{2}\rfloor$ times.
\end{proof}
As already anticipated, our focus is on the maximal part in a partition as in Definition~\ref{def_unref}. In the next section, using Lemma~\ref{lemma:bound}, we show  that the maximal part in an unrefinable partition of $n$ is $O(n^{1/2})$.
\section{Upper bounds on the maximal part}\label{sec:bounds}
It is easy to check that the complete partitions $\pi_1, \pi_2, \dots, \pi_5$ are the only unrefinable partitions for the triangular numbers $T_1, T_2, \ldots, T_5$ respectively. In the general case of $T_n$ for $n\geq 6$, this is not true. For example, the partition $(1,2,3,7,8)\vdash 21 = T_6$ is unrefinable. As a more complex example, in the case of $T_9$ we can calculate that 
\[
\begin{aligned}
&(1, 2, 3, 4, 5, 6, 7, 8, 9 ) &( 1, 2, 3, 5, 6, 7, 10, 11 ) \\
&( 1, 2, 3, 4, 6, 8, 10, 11 ) &( 1, 2, 3, 4, 5, 9, 10, 11 )\\
&( 1, 2, 3, 4, 6, 7, 10, 12  )&( 1, 2, 3, 4, 5, 8, 10, 12 )\\
&(1, 2, 3, 4, 5, 7, 11, 12 ) &( 1, 2, 3, 4, 5, 7, 10, 13 )\\
&( 1, 2, 3, 4, 5, 6, 11, 13 )& ( 1, 2, 3, 4, 5, 6, 10, 14 )\\
&( 1, 2, 4, 5, 8, 11, 14 )&
\end{aligned}
\]
are all the unrefinable partitions of $45=T_9$. 

It is clear that the property of being unrefinable imposes on the one hand an upper limitation on the size of the largest part which is admissible in the partition, and on the other a lower limitation on the minimal excludant.
We address in this section the natural question of determining what is the maximal part in an unrefinable partition of $N$. In the case where $N$ is a triangular number  the following result provides an answer. The notation introduced in the proof will be used throughout the paper.
\begin{proposition}\label{bound_trian}
Let $n \in \mathbb N$ and $N =T_n$. For every unrefinable partition $\listP$ of $N$ we have
\begin{equation}\label{eq:bound_trian}
n \leq \l_t \leq 2n-4.
\end{equation}
Equivalently, 
\[
\dfrac{\sqrt{1+8N}-1}{2}\leq \l_t \leq \sqrt{1+8N}-5.
\]
\end{proposition}
\begin{proof} 
Let us start by considering the complete partition $\pi_n\vdash N$. Other unrefinable partitions of $N$ are obtained from $\pi_n$ by removing some  parts smaller than or equal to $n$ and replacing them 
with parts larger than $n$. Hence, the lower bound for the maximal part in any partition of $N$ is $n$, obtained when no part is removed. Since $N=n(n+1)/2$,  
$n$ is the positive solution of $n^2+n-2N=0$ and so
we have 
\[
\l_t \geq \dfrac{\sqrt{1+8N}-1}{2}.
\]
Let $h,j \in \mathbb N$ and let  us denote by $1 \leq a_1 < a_2< \dots< a_h \leq n$ the candidate parts to be removed from $\pi_n$ to obtain a new unrefinable partition of $N$, and by $n+1 \leq \al_1 < \al_2< \dots< \al_j$
the corresponding replacements. Since $\sum a_i = \sum \al_i$ we have $h > j$. Moreover $j>1$, otherwise from  $\sum a_i = \al_1$ the obtained partition is refinable. Hence we obtain 
\[
h \geq 3,\,\, 
j \geq  2, \text{ and } h > j\;.
\] 
There are $h$ missing parts in the interval $\{1,2,\dots,n\}$ and exactly $j$ parts appear in the interval $\{n+1,n+2,\dots, \al_j\}$. Therefore 
the number of missing parts of $\l$ is
\[
m=h+\al_j-n-j.
\]
To prove $\al_j \leq 2n-4$ we consider the cases where 
$\al_j$ is either equal to $2n-3$, equal to $2n-2$, or strictly larger than $2n-2$. We derive a contradiction in each case.
Let us observe that 
\[
\sum_{i=1}^h a_i \leq n + (n-1) + \dots + (n-(h-1)) = hn - \frac{(h-1)h}{2}.
\]

In the case $\al_j= 2n-3$ we obtain $m=h+n-3-j$. By Lemma~\ref{lemma:bound}, we have $m\leq \left\lfloor \al_j/2\right\rfloor=n-2$, hence 
$h\leq j +1$, and so $h=j+1$. Notice that
\[
\al_1+\dots +\al_j > (j-1)n + 2n-3= (j+1)n-3=hn-3.
\]
Therefore, since $\sum a_i = \sum \al_i$, we have 
\[
3>\frac{(h-1)h}{2},
\]
which is satisfied if $h < 3$, a contradiction.

In the case $\al_j= 2n-2$ we obtain $m=h+n-2-j$. By Lemma~\ref{lemma:bound}, we have $m\leq \left\lfloor \al_j/2\right\rfloor=n-1$, hence 
$h\leq j +1$, and so again $h=j+1$. Notice that
\[
\al_1+\dots +\al_j > (j-1)n + 2n-2= (j+1)n-2=hn-2.
\]
Therefore, since $\sum \al_i = \sum a_i$, we have 
\[
2>\frac{(h-1)h}{2},
\]
which is satisfied if $h < {(1+\sqrt{17})}/{2}< 3$, a contradiction.

To conclude we consider the last case $\al_j> 2n-2$. We have $n-1< {\al_j}/{2}$ and so 
\[
\al_j-(n-1)>  \frac{\al_j}{2}\geq\left\lfloor \dfrac{\al_j}{2}\right\rfloor.
\] 
Hence, since $h\geq j+1$, we have
\[
m=h+\al_j-n-j > \left\lfloor \dfrac{\al_j}{2}\right\rfloor+h-j-1\geq \left\lfloor \dfrac{\al_j}{2}\right\rfloor,
\]
which contradicts Lemma~\ref{lemma:bound}.
\end{proof}
Notice that the upper bound of Eq.~\ref{eq:bound_trian} is tight. Indeed,
let us define the following partition:
\begin{equation}\label{eq:pitilde}
\pt \deq (1,2, \dots, n-3, n+1,2n-4).
\end{equation}
It is easy to notice that $\pt \vdash N$ and that $\pt$ is unrefinable, since its least missing parts are $n-2$ and $n-1$, and 
$2n-4 <(n-2)+(n-1)$. In the notation of the proof of Proposition~\ref{bound_trian}, $\pt$ is obtained in the case $h=3$ and $j=h-1=2$.

The counterpart of Proposition~\ref{bound_trian} in the case of non-triangular numbers is obtained in a similar way.
\begin{proposition}\label{bound_untrian}
Let $N \in \mathbb N$ be such that $T_{n-1} < N < T_n$  for some $n \in \mathbb N$. For every unrefinable partition $\listP$  of $N$ we have
\begin{equation}\label{eq:bound_nontrian}
n \leq \l_t \leq 2n-2.
\end{equation}
Equivalently, 
\[
\dfrac{\sqrt{1+8(N+d)}-1}{2}\leq \l_t \leq \sqrt{1+8(N+d)}-3,
\]
where $d=T_n - N$.
\end{proposition}
\begin{proof}
Let us start by considering $d$ and the partition $\pi_{n,d}\vdash N$ as in Eq.~\eqref{eq:pind}. Other partitions of $N$ are obtained from $\pi_{n,d}$ by removing some parts  smaller than or equal to $n$ 
which are replaced by $d$ and other parts larger than $n$ or only by other parts larger than $n$. Proceeding as in the proof of Proposition~\ref{bound_trian}, let $h,j \in \mathbb N$ and let  us denote by $1 \leq a_1 < a_2< \dots< a_h \leq n$ the candidate parts to be removed from $\pi_{n,d}$ to obtain a new partition of $N$, and by $\al_1< \al_2<\dots<\al_j$ the corresponding replacements. Since $\sum a_i = \sum \al_i$ we have $h \geq j>1$, and we may obtain $h=j$ only if $\al_1=d$.  For this reason, we need to consider the two cases separately.

 Let us assume $\al_i > n$, for every $1\leq i\leq j$. Reasoning as in the proof of Proposition~\ref{bound_trian} we can count
$m=(h+1)+\al_j-n-j$.
On the other hand, if $\al_1=d$ and $\al_i > n$ for every $2\leq i\leq j$, then
we obtain just $h$ missing parts in the interval $\{1,2,\dots,n\}$ and exactly $j-1$ parts appear in the interval $\{n+1,n+2,\dots, \al_j\}$, therefore
 we obtain the same formula for the number of missing parts $m=h+\al_j-n-(j-1)$. By Lemma~\ref{lemma:bound} we obtain
 \begin{equation}\label{eqprop2}
 h+\left\lceil \dfrac{\al_j}{2}\right\rceil -n-j+1 \leq 0.
 \end{equation}
 If $\al_j > 2n-2$, then $\left\lceil {\al_j}/{2}\right\rceil \geq n$ and from Eq.~\eqref{eqprop2} we obtain $h\leq j-1$, a contradiction.
\end{proof}
\begin{remark}
Notice that the bound of Eq.~\eqref{eq:bound_nontrian} is reached by the partition $(1,2,\ldots,n-2,2n-2)\vdash T_n-1$ constructed from 
$\pi_{n,1}=(2,3,\ldots,n)$.
\end{remark}

We now introduce maximal unrefinable partitions, the main subject of this work,  as those partitions $\listP \in \mup{N}$ whose $\l_t$ is maximal. 

\begin{definition}\label{def:maxunref}
Let $N \in \mathbb N$. An unrefinable partition $\listP$ of  $N$ is called \emph{maximal} if  
\[
\l_t = \max_{(\l_1', \l_2', \dots, \l_t') \in \up_{N}} \l_t'.
\]
We denote by ${\mup{N}}$ the set of the maximal unrefinable partitions of $N$.
\end{definition}

In the case of triangular numbers $N =T_n$ for some $n \geq 6$, by virtue of Proposition~\ref{bound_trian}, we have that $\listP \in \mup{N}$ is maximal if and only if $\l_t = 2n-4$. 

\begin{remark}\label{rmk:notehj}
As already observed in the proof of Proposition~\ref{bound_trian}, for each $\l \in \up$ there exist $1 <j< h$, $1\leq a_1< a_2 <\dots < a_h\leq n$ and 
$\al_1, \al_2,\dots, \al_j\geq n+1$ such that $\l$ is obtained from $\pi_n$ by removing the parts $a_i$s  which are replaced by the parts $\al_i$s. Consequently,
$\#\mup{N}$ coincides with the number of choices which lead to partitions meeting the mentioned conditions and, in addition, the condition $\l_t=2n-4$.
In the remainder of the work, when $\l \in \up$ we will refer to $a_i$s, $\al_i$s, $j$ and $h$ as intended here.
\end{remark}


\section{Classification of maximal unrefinable partitions of triangular numbers}\label{sec:maximal}
We are now ready to prove our first main contribution. Using arguments similar to those of the proofs in the previous section, 
 we classify maximal unrefinable partitions for triangular numbers. 

\begin{theorem}\label{th:maxevenodd}
Let $n \in \mathbb N$, $n \geq 6$, and $N =T_n$. Then 
\begin{enumerate}
\item \label{eq1thm} if $n$ is even, then $ \mup{N}= \{\widetilde{\pi_n}\}$;
\item  \label{eq2thm} if $n$ is odd, then $\pt \in \mup{N}$ and the other partitions $\l \in \mup{N}$,  $\l \neq \widetilde{\pi_n}$, are such that $j = h-2$ and the following conditions hold:
\begin{enumerate}
\item[(i)]  the removed parts $a_1,\ldots, a_{h-3}$ are replaced by $2n-4-a_1, 2n-4-a_2, \ldots, 2n-4-a_{h-3}$, and
\item[(ii)] the triple $(a_{h-2},a_{h-1},a_{h})$ is one of the following
\[
(n-4,n-3,n-2),(n-4,n-2,n-1),(n-3,n-2,n),(n-2,n-1,n).
\]
\end{enumerate}
\end{enumerate}
\end{theorem}
\begin{proof}
 Let $\l\in\mup{N}$ 
  and let $a_1,a_2,\ldots, a_h$ and $\al_1,\al_2,\ldots,\al_j=2n-4$  as in Remark~\ref{rmk:notehj}.
We already know that $h\geq 3$. From the hyphotheses on $\l$ we have that \[m=h+\al_j-n-j=h+n-4-j.\] By Lemma~\ref{lemma:bound} we have $h-j\leq 2$, and, since $h>j$, we obtain $j\in\{h-1,h-2\}$. 
Notice that if $a \in \{a_1,\ldots,a_h\}$ is such that $a < n-4$, then $\al \deq 2n-4-a$ must belong to $\{\al_1,\ldots,\al_{j-1}\}$, otherwise $\al+a=2n-4=\al_j\in\l$, and so $\l$ is refinable.
Then each removed part $a_i$ such that $a_i<n-4$ is in one-to-one correspondence with its \emph{replacement} which, for the sake of simplicity, we will denote from now on by $\al_i$. On the other side, for the same symmetry argument, no part in the interval $\{n-4,\dots, n\}$ has a replacement. In such an interval we may choose at 
  most 5 parts. However, 
 we are not allowed to remove, at the same time, parts from one of the pairs $(n-4,n)$ and $(n-3,n-1)$ without contradicting the unrefinability of $\l$. Analogously, we are not allowed to remove more than three parts. Moreover, we cannot choose to pick only one part to be removed in that interval, otherwise we would obtain $h-1$ replacements but 
  at most $j-1$ are allowed, and $h > j$.
  
 We are left to consider the cases of two or three parts to be removed in the interval $\{n-4,\dots, n\}$, both with the assumptions $j = h-1$ or $j = h-2$. In particular, we will show that in both settings of $j$, there exists no maximal partition with two removed parts in the selected interval. Moreover, in the case $j=h-1$ and three removed parts, we show that the only admissible partition is $\pt$. Finally, partitions in the case $j=h-2$ and three removed parts are only possible for odd $n$ as claimed in~\eqref{eq2thm}. Let us address each of the four cases separately.
 
 Let us suppose $j=h-1$ and  $1\leq a_1<a_2<\dots<a_{h-2}\leq n-5$ and $n-4\leq a_{h-1}<a_h\leq n $. For each $1\leq i \leq j-1 = h-2$
 we have $\al_i=\al_j-a_i$. We will now show that this configuration leads to a contradiction. To do this, 
 we estimate $\sum a_i$ and $\sum \al_i$ from above and from below, respectively.
 This is clearly accomplished by noticing that $\al_{h-2}\geq n+1, \al_{h-3}\geq n +2, \ldots, \al_1\geq n+h-2$ and $a_h\leq n, a_{h-1}\leq n-1$, obtaining $a_{h-2}\leq\al_j-\al_{h-2} \leq n-5, a_{h-3} \leq n-6, \ldots, a_{1} \leq n-h-2$. 
 Hence 
\[
\sum_{i=1}^h a_i \leq hn-\sum_{i=1}^{h+2}i +2+3+4=hn-\frac{h^2+5h+6}{2}+9, \mbox{ and}
\]
\[
\sum_{i=1}^{j} \al_i \geq hn+\sum_{i=1}^{h-2}i -4=hn+\frac{h^2-3h+2}{2}-4.
\]
For $\sum a_i = \sum \al_i$ we obtain an inequality which is satisfied for $h< 3$ , which is a contradiction. 

In the second case, i.e.\ $j=h-1$, and $1\leq a_1<a_2<\dots<a_{h-3}\leq n-5$, $n-4\leq a_{h-2}<a_{h-1}<a_h\leq n$, we have  $\al_i=\al_j-a_i$ for every $1\leq i \leq h-3=j-2$. Notice that, in this case, the part $\al_{j-1} = \al_{h-2}$ is not determined by one of the $a_i$s. Proceeding as before, since
 $\al_{h-3}\geq n+1, \al_{h-4} \geq n+2, \ldots, \al_1 \geq n+h-3$, $\al_{h-2} \geq n+h-2$ and $a_h\leq n, a_{h-1} \leq  n-1, a_{h-2} \leq n-2$, we determine
 $a_{h-3} \leq \al_j-\al_{h-3} \leq n-5, a_{h-4} \leq n-6, \ldots, a_{1} \leq n-h-1$ and we obtain the bounds
 \[
\sum a_i \leq hn-\sum_{i=1}^{h+1}i +3+4=hn-\frac{h^2+3h+2}{2}+7, \mbox{ and}
\]
\[
\sum \al_i \geq hn+\sum_{i=1}^{h-2}i -4=hn+\frac{h^2-3h+2}{2}-4.
\]
From $\sum a_i = \sum \al_i$ we obtain an inequality which is satisfied only for $h= 3$, which corresponds to the partition (cf.~also Eq.~\eqref{eq:pitilde})
 \[(1,2,\ldots,n-3,n+1,2n-4)=\widetilde{\pi_n} \in \mup{N}.\]

The third case $j=h-2$ with two removed parts is immediately contradictory, since the parts $a_1, a_2, \dots, a_{h-2}$ determine $h-2=j$ replacements but at most $j-1$ are possible.

The last case to be considered is the one where $j = h-2 $ and the three largest parts $a_i$s are chosen in the interval $\{n-4,\dots, n\}$. As already observed, since $\l$ is unrefinable, the only possible choices are 
\[
(a_{h-2},a_{h-1},a_{h})\in\{(n-4,n-3,n-2),(n-4,n-2,n-1),(n-3,n-2,n),(n-2,n-1,n)\},
\]
which means $a_{h-2}+a_{h-1}+a_h \in \{3n-9, 3n-7, 3n-5, 3n-3\}$. From $\sum a_i = \sum \al_i$ we obtain
\[
a_1+a_2+\dots+a_h=(\al_{h-2}-a_1)+(\al_{h-2}-a_2)+\dots+(\al_{h-2}-\al_{h-3})+\al_{h-2}
\]
and so, since $\al_{h-2}=\al_j = 2n-4$,
\begin{equation}\label{eq:odd}
2(a_1+a_2+\dots+a_{h-3})+(a_{h-2}+a_{h-1}+a_h) = (h-2)(2n-4).
\end{equation}
Since the right side of Eq.~\eqref{eq:odd} is even and $a_{h-2}+a_{h-1}+a_h$ is  even  only if $n$ is odd, Eq.~\eqref{eq:odd} can be satisfied only in the case when $n$ is odd.
This proves \eqref{eq2thm} when $n$ is odd and that the partition $\pt$ of Eq.~\eqref{eq:pitilde} is the only maximal unrefinable partition of $T_n$ when $n$ is even, i.e.~\eqref{eq1thm}. 
\end{proof}
From Theorem~\ref{th:maxevenodd} we obtain that the description of maximal unrefinable partitions for the triangular number of an even integer is completed. The odd case is addressed in the following section.
\begin{corollary}
Let $k \in \mathbb N$ and $N=T_{2k}$. Then $\#\mup{N}= 1$.
\end{corollary}


\section{Triangular numbers of odd integers}\label{sec:odd}
Throughout this last section, $N$ will denote the triangular number of an odd integer. More precisely,
let $n = 2k-1 \in \mathbb N$ be such that $N =T_n$. 

From Theorem~\ref{th:maxevenodd} we have that
the set of maximal unrefinable partitions of triangular numbers of odd integers can be partitioned in the following way
\[
\{\widetilde{\pi_n}\mid n \text{ odd }\}\,\dot\cup\,\mA \,\dot\cup\, \mB \,\dot\cup\, \mC \,\dot\cup\, \mD,
\]
where 
\[
\mA \deq \bigcup_{h \geq 4}\mA_h, \quad \mB \deq \bigcup_{h \geq 4}\mB_h, \quad \mC \deq \bigcup_{h \geq 4}\mC_h, \quad\mD \deq \bigcup_{h \geq 4}\mD_h
\]
and 
\begin{eqnarray*}
\mA_h &\deq&  \bigcup_{n \text{ odd }}\{\l \mid \l \in \mup{T_n}, (a_{h-2},a_{h-1}, a_{h}) = (n-4,n-3,n-2)\}, \\
\mB_h &\deq&  \bigcup_{n \text{ odd }}\{\l \mid \l \in \mup{T_n}, (a_{h-2},a_{h-1}, a_{h}) = (n-4,n-2,n-1)\}, \\
\mC_h &\deq&  \bigcup_{n \text{ odd }}\{\l \mid \l \in \mup{T_n}, (a_{h-2},a_{h-1}, a_{h}) = (n-3,n-2,n)\}, \\
\mD_h &\deq& \bigcup_{n \text{ odd }} \{\l \mid \l \in \mup{T_n}, (a_{h-2},a_{h-1}, a_{h}) = (n-2,n-1,n)\}. 
\end{eqnarray*}
Each set $\mA_h,\mB_h,\mC_h,\mD_h$ is called a \emph{class} of maximal unrefinable partitions. If $\l \in \mA_h$ (resp. $ \mB_h, \mC_h$ or $ \mD_h$) for some $h$ we say that $\l$ is a \emph{partition of class} $\mA_h$ (resp.\ $ \mB_h, \mC_h$ or $ \mD_h$).

The following consideration is a trivial but important consequence of Theorem~\ref{th:maxevenodd}.
\begin{corollary}\label{cor:asym}
Let $n\in\mathbb{N}$ and $N=T_n$. If $\l = (\l_1,\l_2, \dots, \l_t)\in\mup{N}$, then $\l_i\ne n-2$ for every $1\leq i\leq t$. 
\end{corollary}

\begin{remark}[Anti-symmetric property]\label{rmk:anti}
From Theorem~\ref{th:maxevenodd}(\ref{eq2thm}) and Corollary~\ref{cor:asym} we derive that every partition $\l \in \mup{T_n}, \l \ne \pt$, is \emph{anti-symmetric} with respect to $n-2$, i.e.
for $1 \leq a <2n-4$ we have
\[
a \notin \l \iff 2n-4-a \in \l,
\]
provided that $a \neq n-2$.
\end{remark}

\begin{example}\label{exa13}
Let us fix $n=13$. In Tab.~\ref{tab_ex22sym} we have displayed the three different partitions of $\mup{T_{13}} \setminus \{\widetilde{\pi_{13}}\}$, where a black dot means that the corresponding integer is a part and the white dot means otherwise. Disregarding the last part which is fixed to be $2n-4$ due to the maximality constraint, the anti-symmetric property with respect to $n-2$ can be appreciated. Notice also that $\min_{\l \in \mup{T_{13}}}\mex{\l} = 5 = (n-3)/2$ and that $(n-2)-5+1 = 7 = \lceil n/2\rceil$.

\begin{table}[h]
\resizebox{0.93\textwidth}{!}{\begin{minipage}{\textwidth}
\begin{tabular}{cccccccc:cc||c||cc:cccccccc:c}
     &&&&&&&&&&\small{$n-2$}&&\small{$n$}&&&&&&&&&\small{$\l_t$}\\\hline
    1 & 2 & 3 & 4 & 5 & 6 & 7 & 8 & 9 & 10 & 11 & 12 & 13 & 14 & 15 & 16 & 17 & 18 & 19 &20 & 21&22 \\	\hline
   $\bl$ & $\bl$  & $\bl$  & $\bl$  & $\bl$  & $\bl$  & $\circ$ & $\bl$ & $\circ$ & $\circ$ & $\circ$ & $\bl$ & $\bl$ & $\circ$ &$\bl$  & $\circ$ & $\circ$ & $\circ$ & $\circ$ &$\circ$ & $\circ$&  $\bl$\\
  $\bl$  & $\bl$ &$\bl$  &$\bl$  &$\bl$  &$\circ$  &$\bl$  & $\bl$ &$\circ$  &  $\bl$& $\circ$ & $\circ$ & $\bl$ & $\circ$ & $\circ$ & $\bl$ &$\circ$  & $\circ$ & $\circ$ & $\circ$ &  $\circ$&$\bl$\\
       $\bl$ &$\bl$  & $\bl$ & $\bl$ &   $\circ$& $\bl$ & $\bl$ &$\bl$  & $\bl$ &  $\circ$ &  $\circ$ & $\bl$ &  $\circ$ &  $\circ$ &  $\circ$ &  $\circ$ & $\bl$ &  $\circ$ &  $\circ$ &   $\circ$&   $\circ$&$\bl$
\end{tabular} 
\end{minipage}}
\bigskip

\caption{The anti-symmetric property shown on the partitions $\l \in \mup{T_{13}}$, $\l \ne \widetilde{\pi_{13}}$.}
\label{tab_ex22sym}
\end{table}
\end{example}

\begin{example}\label{exa27}
As another significative example, we show in Tab~\ref{tab:exa27} all the partitions in $\mup{T_{27}} = \mup{378}$, classified according to the description
of Theorem~\ref{th:maxevenodd}. It is important to notice that, when $h \geq 5$, partitions in the same class may appear with different multiplicities.
Here all the parts $\l_i$s of the partitions are listed, divided in three \emph{areas} $1 \leq \l_i \leq n-5$, $n-4 \leq \l_i \leq n$ and $n+1 \leq \l_i \leq 2n-4$ naturally induced by Theorem~\ref{th:maxevenodd}. Notice again that we have $\min_{\l \in \mup{T_{27}}}\mex{\l} = 12 = (n-3)/2$. 

\begin{table}[]
\resizebox{0.9\textwidth}{!}{\begin{minipage}{\textwidth}
\begin{tabular}{|l|l|l||c|}
\hline
$ \quad \quad\quad \quad \quad \quad \quad \quad \quad \quad 1 \leq \l_i \leq 22$ & $23\leq \l_i \leq 27$ & $28 \leq \l_i \leq 50$ & \text{class}\\
\hline\hline
1 2 3 4 5 6 7 8 9 10 11 12 13 14 15 16 17 18 19 20 21 22 &23 24 &28 50 &$\widetilde{\pi_{27}}$\\ 
\hline\hline
1 2 3 4 5 6 7 8 9 10 11 12 13 15 16 17 18 19 20 21 22& 26 27 &36 50 &$\mA_4$\\ 
1 2 3 4 5 6 7 8 9 10 11 12 14 15 16 17 18 19 20 21 22 &24 27 &37 50 &$\mB_4$\\ 
1 2 3 4 5 6 7 8 9 10 11 13 14 15 16 17 18 19 20 21 22& 23 26 &38 50 &$\mC_4$\\ 
\hline\hline
1 2 3 4 5 6 7 8 9 10 11 12 13 14 15 16 17 18 21 22 &26 27& 30 31 50 &$\mA_5$\\ 
1 2 3 4 5 6 7 8 9 10 11 12 13 14 15 16 17 19 20 22& 26 27 &29 32 50 &$\mA_5$\\ 
1 2 3 4 5 6 7 8 9 10 11 12 13 14 15 16 18 19 20 21 &26 27 &28 33 50 &$\mA_5$\\ 
1 2 3 4 5 6 7 8 9 10 11 12 13 14 15 16 17 19 21 22& 24 27 &30 32 50 &$\mB_5$\\ 
1 2 3 4 5 6 7 8 9 10 11 12 13 14 15 16 18 19 20 22& 24 27 &29 33 50&$\mB_5$\\ 
1 2 3 4 5 6 7 8 9 10 11 12 13 14 15 17 18 19 20 21& 24 27 &28 34 50 &$\mB_5$\\ 
1 2 3 4 5 6 7 8 9 10 11 12 13 14 15 16 17 20 21 22& 23 26 &31 32 50&$\mC_5$\\ 
1 2 3 4 5 6 7 8 9 10 11 12 13 14 15 16 18 19 21 22 &23 26 &30 33 50 &$\mC_5$\\ 
1 2 3 4 5 6 7 8 9 10 11 12 13 14 15 17 18 19 20 22& 23 26 &29 34 50 &$\mC_5$\\ 
1 2 3 4 5 6 7 8 9 10 11 12 13 14 16 17 18 19 20 21 &23 26 &28 35 50 &$\mC_5$\\ 
1 2 3 4 5 6 7 8 9 10 11 12 13 14 15 16 18 20 21 22 &23 24 &31 33 50 &$\mD_5$\\ 
1 2 3 4 5 6 7 8 9 10 11 12 13 14 15 17 18 19 21 22 &23 24 &30 34 50 &$\mD_5$\\ 
1 2 3 4 5 6 7 8 9 10 11 12 13 14 16 17 18 19 20 22 &23 24& 29 35 50 &$\mD_5$\\ 
\hline\hline
1 2 3 4 5 6 7 8 9 10 11 12 13 14 15 16 17 18 19& 24 27 &28 29 30 50 &$\mB_6$\\ 
1 2 3 4 5 6 7 8 9 10 11 12 13 14 15 16 17 18 20& 23 26 &28 29 31 50 &$\mC_6$\\ 
1 2 3 4 5 6 7 8 9 10 11 12 13 14 15 16 17 18 21& 23 24 &28 30 31 50&$\mD_6$\\ 
1 2 3 4 5 6 7 8 9 10 11 12 13 14 15 16 17 19 20& 23 24 &28 29 32 50 &$\mD_6$\\ 
\hline
\end{tabular}
\end{minipage}}
\bigskip

\caption{Maximal unrefinable partitions of $378=T_{27}$ and the corresponding classes.}	
\label{tab:exa27}
\end{table}
\end{example}
\subsection{Bounds for $h$ and the minimal excludant}\label{sec:props}
It is natural to wonder, recalling that in general  $h \geq 4$, what is an upper bound for $h$ in a maximal unrefinable partition $\l \in \mup{T_{2k-1}}$. The answer to this question is provided in Proposition~\ref{prophgen}, from which we also derive the result on the lower bound for the minimal excludant in maximal unrefinable partitions (cf.~Corollary~\ref{cor:mex}).
Let us address before the two extremal cases $h=4$ and $h=5$.
\begin{proposition}\label{proph4}
Let $n \geq 7$ be odd. We have:
\begin{enumerate}
    \item \label{item1_proph4} $\mup{T_n}\cap \mD_4=\emptyset$,
    \item \label{item2_proph4}$\mup{T_7}\cap \mC_4=\emptyset$ and if $n\geq9$, then   $\#(\mup{T_n}\cap \mC_4)=1$,
    \item \label{item3_proph4}$\mup{T_7}\cap \mB_4=\mup{T_9}\cap \mB_4=\emptyset$ and if $n\geq 11$, then $\#(\mup{T_n}\cap \mB_4)=1$,
    \item \label{item4_proph4}$\mup{T_7}\cap \mA_4=\mup{T_9}\cap \mA_4=\emptyset$  and if $n\geq 11$, then   $\#(\mup{T_n}\cap \mA_4)=1$.
\end{enumerate}
\end{proposition}
\begin{proof}
Let $\l \in \mup{T_n}$ be obtained by removing the integers $a_1, a_2, a_3, a_4$ and adding the replacements $\al_1$ and $\al_2=2n-4$, which need to satisfy the following conditions:
\begin{enumerate}[(i)]
    \item $1\leq a_1\leq n-5$,
    \item $n-4\leq a_2<a_3<a_4\leq n$,
    \item $n+1\leq\alpha_1=2n-4-a_1$,
    \item \label{item_last} $a_1+a_2+a_3+a_4=\al_1+\al_2$.
\end{enumerate}
For a  contradiction, let us assume that $\l \in \mD_4$, and so $a_2=n-2, a_3=n-1$ and $a_4=n$. From Eq.~(\ref{item_last})  we obtain 
\[a_1=\frac{n-5}{2} \quad \text{ and } \quad \al_1=\frac{3n-3}{2}.\]
Notice that $a_1+(n+1)=\al_1$ and, by hypothesis, $\al_1 \geq n+1$. If $\al_1=n+1$, we obtain $n=5$, a contradiction. Otherwise, since $(n+1) \notin \l$, we obtain that $\l$ is refinable. 
The claim~\eqref{item1_proph4} is then proved.

Let us now address the case~\eqref{item2_proph4}. Similarly as before, we now  have  $a_2=n-3, a_3=n-2$ and $a_4=n$ and so we determine
 \[a_1=\frac{n-3}{2} \quad \text{ and } \quad \al_1=\frac{3n-5}{2}.\]
Notice that from $\al_1 \geq n+1$ we obtain $n \geq 7$. However, assuming 
$n=7$ leads to $a_1+a_2=6=n-1 \in \l$, a contradiction since $\l$ is unrefinable.
Let us now prove that the obtained partition
\[
\l \deq \left(\dots,\hat{\frac{n-3}{2}},\dots,n-4,n-1,\frac{3n-5}{2},2n-4\right)
\]
 is unrefinable by showing that each possible sum $a_i+a_j$,
with $1\leq i < j \leq 4$, is different from $\al_1$. Recall that by the classification of Theorem~\ref{th:maxevenodd} we have already 
ruled out those partitions which contradict the unrefinability in $2n-4$.
Since $n\geq 9$, we have that $a_1+a_2={(3n-9)}/{2}>n-1$. Consequently, every sum of missing parts is larger than $n-1\in\l$.
Moreover, $a_1+a_2\neq\al_1$, $a_1+a_3={(3n-7)}/{2}\neq\al_1$,  $a_1+a_4={(3n-3)}/{2}>\al_1$,  $a_2+a_3=2n-5>\al_1$ and therefore $a_2+a_4, a_3+a_4 > \al_1$.
Therefore $\l \in\mC_4$ and it is unique by construction, which proves the claim~\eqref{item2_proph4}.

In the case of $\mB_4$, 
we find
\[a_1=\frac{n-1}{2} \quad \text{ and } \quad\al_1=\frac{3n-7}{2}.\]
From  $\al_1\geq n+1$ we have $n\geq9$, and assuming $n=9$  contradicts again the unrefinability; therefore $n\geq11$. With arguments similar to those of the previous case  the partition 
\[\left(\dots,\hat{\frac{n-1}{2}},\dots,n-5,n-3,n,\frac{3n-7}{2},2n-4\right)\] is proved unrefinable and unique by construction, hence~\eqref{item3_proph4} is obtained.

Finally, considering the case of $\mA_4$, we obtain
\[ a_1=\frac{n+1}{2} \quad \text{ and } \quad \al_1=\frac{3n-9}{2}.\]
Now, $\al_1\geq n+1$ implies $n\geq11$ and $a_1+a_2={(3n-7)}/{2}>\al_1$. This proves that
 \[\left(\dots,\hat{\frac{n+1}{2}},\dots,n-5,n-1,n,\frac{3n-9}{2},2n-4\right)\in\mA_4,\]
 i.e.\ the claim~\eqref{item4_proph4}.
\end{proof}

\begin{proposition}\label{proph5}
Let $n \geq 7$ be odd and $k\geq 0$. We have:
\begin{enumerate}

\item \label{item1_proph5} if $n < 15$, then $\mup{T_n}\cap \mC_5=\emptyset$   and  if $n \geq 15$, then   $\#(\mup{T_{15+2k}} \cap \mC_5) = \lfloor k/2\rfloor+1$,
\item \label{item2_proph5} if $n < 17$, then $\mup{T_n}\cap \mB_5=\emptyset$   and  if $n \geq 17$, then $\#(\mup{T_{17+2k}} \cap \mB_5 )= \lfloor k/2\rfloor+1$,
\item \label{item3_proph5} if $n < 17$, then $\mup{T_n}\cap \mD_5=\emptyset$   and  if $n \geq 17$, then  $\#(\mup{T_{17+2k}} \cap \mD_5) = \lfloor k/2\rfloor+1$,
\item \label{item4_proph5} if $n < 19$, then $\mup{T_n}\cap \mA_5=\emptyset$   and  if $n \geq 19$, then  $\#(\mup{T_{19+2k}} \cap \mA_5) = \lfloor k/2\rfloor+1$.

\end{enumerate}
\end{proposition}
\begin{proof}
Let us proceed as in the proof of Proposition~\ref{proph4}.  Let $\l \in \mup{T_n}$ be obtained by removing the integers $a_1, a_2, \dots, a_5 $ and adding the replacements $\al_1$, $\al_2$ and $\al_3=2n-4$, which need to satisfy the following conditions:
\begin{enumerate}[(i)]
    \item $1\leq a_1<a_2\leq n-5$,
    \item $n-4\leq a_3<a_4<a_5\leq n$,
    \item \label{item_a1a2} $n+1\leq\al_2=2n-4-a_2<\al_1=2n-4-a_1$,
    \item \label{item_last2} $a_1+a_2+a_3+a_4+a_5=\al_1+\al_2+\al_3$.
\end{enumerate}
First, let us  address the case~\eqref{item1_proph5}. We  have  $a_3=n-3, a_4=n-2$ and $a_5=n$ and, from Eq.~\eqref{item_a1a2} and Eq.~\eqref{item_last2}, $a_1$ and $a_2$ satisfy the
condition  
 \begin{equation}\label{proph5:eq1}
 a_1+a_2=\frac{3n-7}{2}.
 \end{equation}
  We first consider the case when $a_2$ is maximal, i.e.\ $a_2=n-5$, in which we have $a_1=(n+3)/2$ and consequently $\al_1 = (3n-11)/2$ and $\al_2= n+1$.
  Notice then that the condition of Eq.~\eqref{proph5:eq1} can be met in $\lfloor(a_2-a_1-1)/2\rfloor$ other ways by taking the first two parts to be removed as
   $a_1+i$ and $a_2-i$, for $1\leq i \leq \lfloor(a_2-a_1-1)/2\rfloor$. Now, from $a_1 <a_2$ we obtain $n\geq 15$. If $n=15$, the partition
\[
\left(\dots,\hat{\frac{n+3}{2}},\dots,\hat{n-5},n-4,n-1,n+1,\frac{3n-11}{2},2n-4\right)
\]
is unique by construction and is unrefinable since $a_1+a_2 > \al_1$.
In the other cases, which are
\begin{equation}\label{proph5:eq2}
\left\lfloor\dfrac{a_2-a_1-1}{2}\right\rfloor = \left\lfloor\dfrac{n-15}{4} \right\rfloor,
\end{equation} 
 we obtain an unrefinable partition since, letting $\al_1'= 2n-4-(a_1+i)$ and $\al_2'= 2n-4-(a_2-i)$, we have 
\[
(a_1+i)+(a_2-i) = a_1+a_2 > \al_1 > \al_1' >\al_2'.
\]
The claim~\eqref{item1_proph5} is then obtained writing $n=15+2k$ in Eq.~\eqref{proph5:eq2}.

The proofs for~\eqref{item2_proph5} and~\eqref{item4_proph5}  are obtained in the same way. When $n=17$, the partition 
\[
\left(\dots,\hat{\frac{n+5}{2}},\dots,\hat{n-5},n-3,n,n+1,\frac{3n-13}{2},2n-4\right)\in\mB_5
\]
and is unique by construction, and when $n > 17$ it can be modified in $\lfloor(a_2-a_1-1)/2\rfloor = \lfloor(n-17)/4\rfloor$ ways as in the
proof of~\eqref{item1_proph5}.
Analogously, when $n=19$ the partition
\[
\left(\dots,\hat{\frac{n+7}{2}},\dots,\hat{n-5},n-1,n,n+1,\frac{3n-15}{2},2n-4\right)\in\mA_5
\]
and is unique by construction, and when $n > 19$ it can be modified in $\lfloor(n-19)/4\rfloor$ ways.

It remains to prove the slightly different case~\eqref{item3_proph5}. Here, 
we have 
$a_1+a_2=(3n-9)/2$ and, proceeding as above, from $a_2=n-5$ we obtain $a_1=(n+1)/2$ and $\al_1 = (3n-9)/2$. 
This leads to the contradiction $a_1+a_2 = \al_1$. The argument of~\eqref{item1_proph5} is here replicated 
starting from $a_2= n-6$.
It is now easy to see that, when $n=17$, the partition
\[
\left(\dots,\hat{\frac{n+3}{2}},\dots,\hat{n-6},\dots,n-3,n+2,\frac{3n-11}{2},2n-4\right),
\]
unique by construction, is unrefinable.  When $n > 17$, it can be modified in $\lfloor(n-17)/4\rfloor$ ways, which proves~\eqref{item3_proph5}.
\end{proof}
 
 \begin{proposition}\label{prophgen}
Let $n\geq7$ be odd and $h\geq6$. We have
\begin{enumerate}
\item \label{item1_proph6} $\mup{T_n}\cap\mD_h \neq \emptyset$ if and only if $n \geq h^2-h-7$,
\item \label{item2_proph6} $\mup{T_n}\cap\mC_h \neq \emptyset$ if and only if $n\geq h^2-h-5$,
\item \label{item3_proph6} $\mup{T_n}\cap\mB_h \neq \emptyset$ if and only if $n \geq h^2-h-3$,
\item \label{item4_proph6} $\mup{T_n}\cap\mA_h \neq \emptyset$ if and only if $n \geq h^2-h-1$.
\end{enumerate}
 \end{proposition}
 \begin{proof}
 We proceed as in Proposition~\ref{proph4} and Proposition~\ref{proph5}, assuming the conditions 
 \begin{enumerate}[(i)]
    \item $1\leq a_1<a_2< \dots<a_{h-3}\leq n-5$,
    \item $n-4\leq a_{h-2}<a_{h-1}<a_h\leq n$,
    \item \label{item2_a1a2}$n+1\leq\al_{h-3}=2n-4-a_{h-3}<\al_{h-4}=2n-4-a_{h-4}<\dots<\al_1=2n-4-a_1$,
    \item \label{item2_last2}$\sum a_i=\sum \al_i$.
\end{enumerate}
If $\l\in \mup{T_n}\cap\mD_h$, then $a_{h-2}+a_{h-1}+a_h=3n-3$ and therefore, from Eq.~\eqref{item2_a1a2} and Eq.~\eqref{item2_last2},
\[
a_1+a_2+\dots+a_{h-3}=\frac{(h-2)(2n-4)-(3n-3)}{2}=\frac{(2h-7)n+11-4h}{2}.
\]
Let us now assume that $a_{h-3}=n-5, a_{h-4}=n-6,\dots, a_2=n-h$, i.e.\ let us maximize the sum $a_2+\dots+a_{h-3}$. 
We obtain \[a_2+\dots +a_{h-3}=(h-4)n-\sum_{i=5}^h i=(h-4)n-\frac{h(h+1)}{2}+10,\] from which we can calculate 
\[
a_1=\frac{n+h^2-3h-9}{2}.
\]

Imposing $a_1<a_2$ we obtain $n>h^2-h-9$. In this setting, we have 
\[
\al_1=\frac{3n-h^2+3h+1}{2} \quad\text{ and } \quad a_1+a_2=\frac{3n+h^2-5h-9}{2}.
\]
 Notice 
that $a_1+a_2 > \al_1$ is satisfied for $h \geq 6$, hence the provided construction leads
to a partition $\l$ which belongs to $\mD_h$ if and only if $n\geq h^2-h-7$, i.e.~\eqref{item1_proph6}.

In the cases \eqref{item2_proph6}, \eqref{item3_proph6} and \eqref{item4_proph6} we proceed analogously, maximizing $a_2+a_3+\dots a_{h-3}$, provided that $a_{h-2}, a_{h-1},a_h$ are modified accordingly.
In particular, when considering $\mC_h$ we have
\[a_1=\frac{n+h^2-3h-7}{2} \quad \text{ and }\quad \al_1=\frac{3n-h^2+3h-1}{2}.\]
From $a_1<a_2$ we have $n\geq h^2-h-5$ and 
from \[
\al_1 < a_1+a_2=\frac{3n+h^2-5h-7}{2}\]
we obtain $\l \in \mC_h$, i.e.\ the claim \eqref{item2_proph6} follows.

In the case of $\mB_h$ we have
\[a_1=\frac{n+h^2-3h-5}{2} \quad \text{ and }\quad \al_1=\frac{3n-h^2+3h-3}{2}.\]
From $a_1<a_2$ we have $n\geq h^2-h-3$ and 
from \[
\al_1 < a_1+a_2=\frac{3n+h^2-5h-5}{2}\]
we obtain $\l \in \mB_h$, i.e.\ the claim \eqref{item3_proph6} is proved.

Finally, assuming the conditions of $\mA_h$ we have
\[a_1=\frac{n+h^2-3h-3}{2} \quad \text{ and }\quad \al_1=\frac{3n-h^2+3h-5}{2}.\]
From $a_1<a_2$ we have $n\geq h^2-h-1$ and 
from \[
\al_1 < a_1+a_2=\frac{3n+h^2-5h-3}{2}\]
we obtain $\l \in \mA_h$, from which the desired result \eqref{item4_proph6} follows.
\end{proof}

By interchanging the role of $n$ and $h$ in the statements of Proposition~\ref{prophgen}, we obtain the following description 
of the set of maximal unrefinable partitions of triangular numbers of an odd number, where 
we can read the upper bound for $h$ in each different class.
\begin{corollary}
Let $n\geq7$ be odd. Then
\[
\mup{T_n}=\{\widetilde{\pi_n}\}\,\dot\cup\left(\bigcup_{h=5}^{\lfloor\frac{1+\sqrt{29+4n}}{2}\rfloor}\mD_h\,
						  \, \dot\cup\bigcup_{h=4}^{\lfloor\frac{1+\sqrt{21+4n}}{2}\rfloor}\mC_h\,
						   \,\dot\cup\bigcup_{h=4}^{\lfloor\frac{1+\sqrt{13+4n}}{2}\rfloor}\mB_h\,
						   \,\dot\cup\bigcup_{h=4}^{\lfloor\frac{1+\sqrt{5+4n}}{2}\rfloor}\mA_h\right)\cap \mup{T_n}.
\]
\end{corollary}

\begin{remark}\label{rmk_forinje}
In the proof of Proposition~\ref{prophgen} we have exhibited an example of unrefinable partition for each class, constructed by maximizing $a_2+a_3+\dots a_{h-3}$ and consequently 
by determining $a_1$. The unrefinability of the obtained partition is then granted from the fact that $a_1+a_2 > \al_1$. 
Notice that, each other partition $\l'$ of the same class is determined by the removed parts $a_1', a_2', \dots, a_{h-3}'$ such that
$a_1'=a_1+i$ and  $a_s'=a_s-i_{s-1}$ for $s > 1$ and $i_s \geq 0$,  where $i=\sum_{s=1}^{h-4}i_s$, provided that $a_i' < a_s'$ for $i<s$. The unrefinability of $\l'$ is then easily proved, since 
\[a_1'+a_2' = a_1+i+a_2-i_i\geq a_1+a_2\geq \al_1 \geq \al_1'. \]
\end{remark}

\begin{example}\label{exa49}
Let $n=49$. For the bound in the previous corollary, when considering partitions of class $\mD$ we have $5 \leq h \leq (1+\sqrt{29+4n})/{2}=8$.
Let us fix $h=7$ and construct all the partitions in $\mup{T_{49}} \cap \mD_7$. We recall that, for Theorem~\ref{th:maxevenodd},  a partition of class $\mD_7$ is given when $a_1, a_2 \dots, a_{h-3}=a_4 $ are specified. Therefore, for the sake of simplicity, we denote the partitions just by listing the removed parts $(a_1,a_2,a_3,a_4)$.
Let us start, as in Proposition~\ref{prophgen}, from the partition
\[
\left(\frac{n+h^2-3h-9}{2}, n-7,n-6,n-5\right)=\left(34,42,43,44\right).
\]
All the remaining partitions in $\mD_7$, obtained as in Remark~\ref{rmk_forinje}, are:
\[
\begin{aligned}
&\left(35,41,43,44\right)  &\left(36,40,43,44\right)\\
&\left(37,39,43,44\right)  &\left(36,41,42,44\right)\\
&\left(37,40,42,44\right)  &\left(38,39,42,44\right)\\
&\left(38,40,41,44\right)  &\left(37,41,42,43\right)\\
&\left(38,40,42,43\right)  &\left(39,40,41,43\right)
\end{aligned}
\]
The partitions in other classes are obtained analogously.
\end{example}

We have already highlighted in Example~\ref{exa13} and in Example~\ref{exa27} what $\min_{\l \in \mup{T_{n}}}\mex{\l}$ looks like. The intuition can now be easily proved 
as a consequence of the previous propositions.
\begin{corollary}\label{cor:mex}
 Let $n\geq 7$  be odd. For each $\l\in\mup{T_n}$ we have \[\mex{\l} = \mu_1 \geq \dfrac{(n-3)}{2}.\]
\end{corollary}
\begin{proof}
Notice that $\mu_1=a_1$.
The claim is trivial if $\l = \pt$. Otherwise it follows from Propositions~\ref{proph4},~\ref{proph5} and~\ref{prophgen}, recalling that $a_1$ was calculated in order to be minimal, since $a_2+a_3+\dots+a_{h-3}$ was maximized. The results are summarized in Table~\ref{tab:recapa1}, where it is not hard to check that ${(n-3)}/{2}$ is the smaller value that $a_1$ can assume.
\begin{table}[]
\begin{tabular}{|c|c|}
\hline
class & $a_1$ \\\hline\hline
$\mC_4$ & ${(n-3)}/{2}$\\
$\mB_4$ & ${(n-1)}/{2}$ \\
$\mA_4$& ${(n+1)}/{2}$\\\hline\hline
$\mD_5$& ${(n+3)}/{2}$\\
$\mC_5 $& ${(n+3)}/{2}$\\
$\mB_5$& ${(n+5)}/{2}$\\
$\mA_5 $& ${(n+7)}/{2}$\\\hline\hline
$\mD_h$& ${(n+h^2-3h-9)}/{2}$\\
$\mC_h $& ${(n+h^2-3h-7)}/{2}$\\
$\mB_h$& ${(n+h^2-3h-5)}/{2} $\\
$\mA_h$& ${(n+h^2-3h-3)}/{2}$\\
\hline
\end{tabular}
\bigskip 

\caption{Values of $a_1$ in the construction of Propositions~\ref{proph4},~\ref{proph5} and~\ref{prophgen}, for $h=4$, $h=5$ and $h \geq 6$.}
\label{tab:recapa1}

\end{table}

\end{proof}

\subsection{The bijection}
 In this conclusive section we prove the main contribution of this work, i.e.\ we show that, when $n$ is odd, the number of maximal unrefinable partitions of $T_n$  equals the number of partitions
of $\lceil n/2\rceil$ into distinct parts by means of a bijective proof. Notice that, by the anti-symmetric property~(Remark~\ref{rmk:anti}) and by the bound on the minimal excludant (Corollary~\ref{cor:mex}), a partition in $\mup{T_n}$ is determined by at most \[(n-2)-\dfrac{n-3}{2} = \dfrac{n+1}{2} -1= \left\lceil \dfrac{n}{2}\right\rceil-1 \] parts. The following theorem is used to establish a bijection between $\mup{T_{2k-1}}$ and $\dist_k$.
\begin{theorem}\label{thm:unique}
Let $a_1, a_2, \dots, a_u$ be the missing parts smaller than or equal to $n-3$ of an unrefinable partition $\l \ne \pt$ of $T_n$, for some odd integer $n \geq 7$. Then $n$, $\l$ and its class are uniquely determined. 
\end{theorem}
\begin{proof}
Let us start by proving that $n$ can be obtained from knowing $a_1, a_2, \dots, a_u$. 
In particular, let us prove that
\begin{equation}\label{formula_n}
n=\frac{2\sum_{i=1}^u a_i+1+4u}{2u-1}
\end{equation}
by distinguishing the four possible classes.
Let us first assume $\l \in \mD_h$ for some $h\geq 5$. Recalling that $a_u \leq n-3$ and, by the definition of $\mD_h$ and by Remark~\ref{rmk:anti}, since $a_{h-1} = n-1$ and $a_h=n$, we have $a_u \notin \{n-3,n-4\}$.
Therefore $a_u \leq n-5$, and so $u=h-3$.
Recalling that the following conditions hold
 \begin{enumerate}[(i)]
    \item $1\leq a_1<a_2< \dots<a_{h-3}\leq n-5$,
    \item $n-4\leq a_{h-2}<a_{h-1}<a_h\leq n$,
    \item \label{item2_a1a2}$n<\al_{h-3}=2n-4-a_{h-3}<\al_{h-4}=2n-4-a_{h-4}<\dots<\al_1=2n-4-a_1$,
    \item \label{item2_last2}$\sum a_i=\sum \al_i$,
\end{enumerate}
we obtain
\[
\sum_{i=1}^u a_i=\frac{(u+1)(2n-4)-(3n-3)}{2},
\]
from which we determine $n$ as claimed.

Let us consider the class $\mC_h$. In this case, reasoning as above, we have $a_u =n-3$ and $a_{u-1} \leq n-5$, which means $u-1=h-3$. Therefore
\[
\sum_{i=1}^u a_i -(n-3)=\sum_{i=1}^{u-1}a_i=\frac{u(2n-4)-(3n-5)}{2},
\]
from which we obtain again Eq.~\eqref{formula_n}.

When $\l \in \mB_h$, we have $a_u=n-4$, which means $h=u+2$, so
\[
\sum_{i=1}^u a_i-(n-4)=\sum_{i=1}^{u-1}a_i=\frac{u(2n-4)-(3n-7)}{2},
\]
from which the same $n$ is determined. 

In conclusion, if $\l \in \mA_h$, we have $a_u = n-3$ and $a_{u-1}=n-4$, so $h=u+1$ and Eq.~\eqref{formula_n} is satisfied since
\[
\sum_{i=1}^{u}a_i-(n-4)-(n-3)=\sum_{i=1}^{u-2}a_i=\frac{(u-1)(2n-4)-(3n-9)}{2}.
\]

Now that $n$ is determined from $a_1, a_2, \dots, a_u$, the class of the partition can be recognized by looking at $a_{u-1}$ and $a_u$. In particular
\begin{itemize}
\item $a_u< n-4 \iff \l \in \mD_{u+3}$,
\item $a_u =  n-3$ and $a_{u-1} <  n-4 \iff \l \in\mC_{u+2}$,
\item $a_u =  n-4 \iff \l \in\mB_{u+2}$,
\item $a_u =  n-3$ and $a_{u-1} =  n-4 \iff \l \in\mA_{u+1}$.
\end{itemize}
To conclude, we determine the partition by using the anti-symmetric property (cf.~Remark~\ref{rmk:anti}).
\end{proof} 
We are now ready to prove our last result. 
Denoting by $\dist$ the set of all the partitions  into distinct parts, let us define the following subsets of $\dist$:
 \begin{eqnarray*}
\mdA{t} &\deq& \{\listP \mid \l \in \dist, \,\, \l_1 = 1, \l_2 = 2, t \geq 3 \}, \\
\mdB{t} &\deq& \{\listP \mid \l \in \dist, \,\, \l_1 = 2, t \geq 2 \}, \\
\mdC{t}&\deq& \{\listP \mid \l \in \dist, \,\, \l_1 = 1, \l_2 > 2, t \geq 2 \}, \\
\mdD{t}&\deq& \{\listP \mid \l \in \dist, \,\, \l_1 > 3, t \geq 2 \}.
 \end{eqnarray*}
 It is not hard to notice that 
 \[
\dist = \mdA{} \,\dot\cup\, \mdB{} \,\dot\cup\, \mdC{}\,\dot\cup\, \mdD{},
\]
where 
\[
\mdA{} \deq \bigcup_{t \geq 3}\mdA{t}, \quad \mdB{} \deq \bigcup_{t \geq 2}\mdB{t}, \quad \mdC{} \deq \bigcup_{t \geq 2}\mdC{t}, \quad\mdD{} \deq \bigcup_{t \geq 2}\mdD{t}.
\]

Let us conclude the paper by proving a bijection between $\mup{T_{2k-1}}$ and $\dist_k$.

\begin{theorem}\label{magari}
Let $k \in \mathbb N$, $k \geq 7$, $n=2k-1$ and  $N=T_{n}$.
Let  $\sigma\colon\mup{T_n}\to \dist_k$ be such that
 \begin{equation*}
\l \mapsto
 \begin{cases}
 \left(3,k-3\right) & \l = \pt \\
 \left(n-2-a_u,\dots,n-2-a_2, n-2-a_1\right) & \l\ne\pt
 \end{cases},
  \end{equation*}
where, if $\l \ne\pt$, then $(a_1, a_2, \dots, a_u)$ are the missing parts of $\Ml \cap \{1,2,\dots,n-3\}$ as in Theorem~	\ref{thm:unique}.
Then $\sigma$ is bijective, therefore
$\#\mup{T_{2k-1}}= \#\dist_k$.
\end{theorem}
\begin{proof}
For the sake of brevity and by virtue of Theorem~\ref{thm:unique}, we will denote each $\l \in \mup{T_n}\setminus\{\pt\}$ by listing its missing parts $a_1, a_2, \dots, a_u$ in $\Ml \cap \{1,2,\dots,n-3\}$.
We prove that $\sigma$ is bijective by proving explicitly that partitions of $\mA\cap \mup{T_{2k-1}}$ are in one-to-one correspondence with those of $\mdA{}\cap \dist_k$
and that the same holds respectively for $\mB$ and $\mdB{}$, $\mC$ and $\mdC{}$, and $\mD \cup \{\pt\}$ and $\mdD{}$. 

Let us start by proving that $\sigma$ is well defined, i.e.\ for each $\l \in \mup{T_n}$ we have that $\sigma(\l)$ is a partition of $k$ into distinct parts. If $\l = \pt$ there is nothing to prove, otherwise, since the missing parts of $\l$ are distinct, so are the parts $ n-2-a_u <\dots<n-2-a_2<n-2-a_1$ of $\sigma(\l)$.
We now prove that the sum of the parts of $\sigma(\l)$ is $k$ in each possible case, making extensive use of Proposition~\ref{proph4}, Proposition~\ref{proph5} and Proposition~\ref{prophgen} without further mention.

\noindent If $\l\in\mA_4$, then $\l=\left(({n+1})/{2},n-4,n-3\right)$ and so
    \[
    \sigma(\l)=\left(1,2,n-2-\frac{n+1}{2}\right)=\left(1,2,\frac{n-5}{2}\right)=\left(1,2,\frac{n+1}{2}-3\right)=(1,2,k-3) \in \dist_k.
    \]
    Notice that, in particular, $\sigma(\l) \in \dist_k \cap\mdA{3}$.
    Similarly, if $\l\in\mB_4$, then $\l=\left({(n-1)}/{2},n-4\right)$ and $\sigma(\l)=\left(2,{(n+1)}/{2}-2\right)=(2,k-2) \in \dist_k\cap\mdB{2}$. If $\l\in\mC_4$ we have $\l=\left({(n-3)}/{2},n-3\right)$, so
    $\sigma(\l)=\left(1,{(n+1)}/{2}-1\right)=(1,k-1) \in \dist_k \cap \mdC{2}$.

\noindent If $\l\in\mA_5$, then $\l=\left((n+7)/{2}+i,n-5-i,n-4,n-3\right)$,  for $0\leq i \le \lfloor (n-19)/4 \rfloor$, and
\[
    \sigma(\l)
     =\left(1,2,3+i,\frac{n-11}{2}-i\right)
    =\left(1,2,3+i,\frac{n+1}{2}-6-i\right)
    =(1,2,3+i,k-6-i)\in \dist_k,
\]
for $0\leq i \le \lfloor (n-19)/4 \rfloor$. In particular, $\sigma(\l) \in \dist_k \cap\mdA{4}$.   Similarly, if $\l\in\mB_5$, then $\l=\left((n+5)/{2}+i,n-5-i,n-4\right)$ and 
\[
\sigma(\l)=\left(2,3+i,\frac{n+1}{2}-5-i\right)=(2,3+i,k-5-i) \in \dist_k\cap\mdB{3},
\] 
for $0\leq i \le \lfloor (n-17)/4 \rfloor$. If $\l\in\mC_5$, we have $\l=\left((n+3)/{2}+i,n-5-i,n-3\right)$ and so
    \[
    \sigma(\l)=\left(1,3+i,\frac{n+1}{2}-4-i\right)=(1,3+i,k-4-i) \in \dist_k \cap \mdC{3},
    \]
for $0\leq i \le \lfloor (n-15)/4 \rfloor$. In the case when $\l=\left((n+3)/{2}+i,n-6-i\right)\in\mD_5$, we have 
  \[
  \sigma(\l)=\left(4+i,\frac{n+1}{2}-4-i\right)=(4+i,k-4-i) \in \dist_k \cap \mdD{2},
  \]
for $0\leq i \le \lfloor (n-17)/4 \rfloor$.\\
Let us now consider $\l\in\mD_h$ for $h\geq6$. In this case 
\[
\l=\left((n+h^2-3h-9)/{2}+i,n-h-i_1,\dots,n-5-i_{h-4}\right),
\] 
where $i=\sum_{s=1}^{h-4}i_s$, and so
    \[
     \begin{aligned}
    \sigma(\l)=&\left(3+i_{h-4},\dots,h-2+i_1,\frac{n-h^2+3h+5}{2}-i\right)\\
    =&\left(3+i_{h-4},\dots,h-2+i_1,k+2-\frac{h^2-3h}{2}-i\right).
     \end{aligned}
    \]
    Notice that
    
    \begin{align*}
   (3+i_{h-4})+(4+i_{h-5})+\dots+(h-2+i_1)&+\\
  +\left(k+2-\frac{h^2-3h}{2}-i\right)&= \\
   \sum_{j=3}^{h-2}j+\sum_{s=1}^{h-4}i_s+k+2-\frac{h^2-3h}{2}-i&=\\
   \frac{(h-2)(h-1)}{2}-3+k+2-\frac{h^2-3h}{2}&=\\
   \frac{h^2-3h}{2}-2+k+2-\frac{h^2-3h}{2}&=k,
    \end{align*}
    
and so $\sigma(\l)\in \dist_k \cap \mdD{h-3}$. Similarly, if  $\l\in\mC_h$, then 
\[
\l=\left(\frac{n+h^2-3h-7}{2}+i,n-h-i_1,\dots,n-5-i_{h-4},n-3\right)
\] 
and
\[
    \sigma(\l)=\left(1,3+i_{h-4},\dots,h-2+i_1,k+1-\frac{h^2-3h}{2}-i\right)\in \dist_k \cap \mdC{h-2}
\]
 since  the sum of the first $h-3$ terms is $(h^2-3h)/{2}-1+\sum_{s=1}^{h-4}i_s$.   If $\l\in\mB_h$, we have 
 \[
 \l=\left(\frac{n+h^2-3h-5}{2}+i,n-h-i_1,\dots,n-5-i_{h-4},n-4\right)
 \] 
 an so
    \[
    \sigma(\l)=\left(2,3+i_{h-4},\dots,h-2+i_1,k-\frac{h^2-3h}{2}-i\right)\in \dist_k \cap \mdB{h-2},
    \]
     since  the sum of the first $h-3$ terms is $(h^2-3h)/{2}+\sum_{s=1}^{h-4}i_s$. Finally, in the case when 
     \begin{equation}\label{eqAh}
     \l=\left(\frac{n+h^2-3h-3}{2}+i,n-h-i_1,\dots,n-5-i_{h-4},n-4,n-3\right)\in\mA_h,
     \end{equation}
     we obtain
     \begin{equation}\label{eqAh2}
     \sigma(\l)=\left(1,2,3+i_{h-4},\dots,h-2+i_1,k-1-\frac{h^2-3h}{2}-i\right)\in\dist_k \cap \mdA{h-1},
     \end{equation}
 noticing that the sum of the first $h-2$ terms is $(h^2-3h)/{2}+1+\sum_{s=1}^{h-4}i_s$.
 
We proved that $\sigma$ is well defined. Notice also that $\sigma$ is trivially injective. Therefore it remains to prove that $\sigma$ is surjective. In particular, it suffices to check that for each partition  $\l^*\in(\mdA{} \,\dot\cup\, \mdB{} \,\dot\cup\, \mdC{}\,\dot\cup\, \mdD{})\cap \dist_k$, $\l^* \ne (3,k-3)$, there exists $\l\in (\mA \,\dot\cup\, \mB \,\dot\cup\, \mC \,\dot\cup\, \mD)\cap \mup{T_{2k-1}}$ such that $\sigma(\l)=\l^*$, since $\sigma(\pt)=(3,k-3)$ by definition. Given 
$\l^* = (\l_1^*,\l_2^*, \dots, \l_t^*)\in \dist_k$, by the definition of $\sigma$ we have that the partition $\l$ denoted by its missing parts  $(n-2-\l_t^*, \dots, n-2-\l_2^*, n-2-\l_1^*)$ is such that $\sigma(\l)=\l^*$. It remains to prove that such $\l$ is a maximal unrefinable partition of $n$.
The full details of the proof are here omitted since they can be obtained by arguments very similar to those used for proving that $\sigma$ is well defined. 
As an example, let us consider the case when $\l^*\in\mdA{t} \cap \dist_k$, for $t\geq 5$, and let us prove that $\l^*$ is the image of an unrefinable partition $\l$ of class $\mA$. 
Since $\l^*$ is a partition of $k$ into $t$ distinct parts and contains $1$ and $2$ by definition we can write
 \begin{equation}\label{eq_lstar}
 \l^*=\left(1,2,3+i_1,\dots,t-1+i_{t-3},k-\sum_{s=1}^{t-1}\l^*_s\right),
 \end{equation}
 where \[
 \sum_{s=1}^{t-1}\l^*_s=\dfrac{(t-1)t}{2}+\sum_{s=1}^{t-3}i_s\] 
 for some $i_1, i_2, \dots, i_{t-3} \geq 0$ (cf.~also Eq.~\eqref{eqAh2}). 
 We can now substitute $t-1$ to $h-2$ and $(n+1)/2$ to $k$ in Eq.~\eqref{eq_lstar}. Applying the correspondence 
 $\l_i \leftrightarrow n-2-\l_{t-i+1}^*$ and 
 denoting the obtained partition $\l$ by listing its missing parts, we obtain 
 \[
 \l=\left(\frac{n+h^2-3h-3}{2}+i,n-h-i_{h-4},\dots,n-5-i_{1},n-4,n-3\right)
 \]
 as in Eq.~\eqref{eqAh}. This proves that $\l \in \mup{T_{2k-1}}\cap \mA_{t+1}$
  is unrefinable (cf.~Proposition~\ref{prophgen} and Remark~\ref{rmk_forinje}) and 
such that $\sigma(\l) = \l^*$. The remaining cases are similar. 
\end{proof}
\begin{remark}
The bijection $\sigma$  is not well defined when $k < 7$. However, it can be easily shown that the result of Theorem~\ref{magari} is still valid 
when $k=4$ and $k=5$, where we have $\#\mup{T_{7}}= \#\dist_4 = 1$ and $\#\mup{T_{9}}= \#\dist_5 = 2$, respectively. The claim is false instead in the case $k=6$, where 
we have $\#\mup{T_{11}} = 4$ and $\#\dist_6 = 3$.
\end{remark}

In the proof of Theorem~\ref{magari} we showed that $\sigma$ is a bijection from $\mup{T_{2k-1}}$ to $\dist_k$. 
Moreover, we also proved that $\sigma$ is bijective when it is restricted to each class.
\begin{corollary}
The function $\sigma$ of Theorem~\ref{magari} maps in a bijective way
\begin{enumerate}[(i)]
\item $\mA_{h}\cap \mup{T_{2k-1}}$ to $\mdA{h-1}\cap \dist_k$, 
\item $\mB_{h}\cap \mup{T_{2k-1}}$ to $\mdB{h-2}\cap \dist_k$,  
\item $\mC_{h}\cap \mup{T_{2k-1}}$ to $\mdC{h-2}\cap \dist_k$,  
\item $\mD_{h}\cap \mup{T_{2k-1}}$ to $\left(\mdD{h-3}\setminus \{(3,k-3)\}\right)\cap \dist_k$.
\end{enumerate}
\end{corollary}

\begin{example}
Coming back to the case of Example~\ref{exa13}, we represent in Tab.~\ref{tab_ex22symcol} the bijection $\sigma$ between maximal unrefinable partitions of $13$ obtained in the case $h = j-2$ (hence those different from $\widetilde{\pi_{13}}$), represented by black dots, and the partitions of $7$ into distinct parts, represented by blue dots. Notice that the partition $(3,4)$ is not displayed since it corresponds to $\widetilde{\pi_{13}}$. Here $x$ corresponds to 
\[
x \deq \min_{\l \in \mup{T_{13}}}\mex{\l}.
\]
Equivalently, by the anti-symmetric property, partitions of $7$ into distinct parts can be read looking at the black dots on the right side of the table.
\begin{table}[]
\resizebox{0.93\textwidth}{!}{\begin{minipage}{\textwidth}
\begin{tabular}{cccccccccc||c||cccccccccc:c}
     &&&&x&&&&&&\small{$n-2$}&&\small{$n$}&&&&&&&&&\small{$\l_t$}\\\hline
    1 & 2 & 3 & 4 & 5 & 6 & 7 & 8 & 9 & 10 & 11 & 12 & 13 & 14 & 15 & 16 & 17 & 18 & 19 &20 & 21&22 \\	
     &  &  &  & \hl{6} & \hl{5} &\hl{ 4} & \hl{3} &\hl{2} & \hl{1} &  &  &  &  &  &  &  &  &  & & & \\	\hline
   $\bl$ & $\bl$  & $\bl$  & $\bl$  & $\bl$  & $\bl$  & $\hl{\bl}$ & $\bl$ & $\hl{\bl}$ & $\hl{\bl}$ & $ $ & $\bl$ & $\bl$ & $ $ &$\bl$  & $ $ & $ $ & $ $ & $ $ &$ $ & $ $&  $\bl$\\
  $\bl$  & $\bl$ &$\bl$  &$\bl$  &$\bl$  &$\hl{\bl}$  &$\bl$  & $\bl$ &$\hl{\bl}$  &  $\bl$& $ $ & $ $ & $\bl$ & $ $ & $ $ & $\bl$ &$ $  & $ $ & $ $ & $ $ &  $ $&$\bl$\\
       $\bl$ &$\bl$  & $\bl$ & $\bl$ &   $\hl{\bl}$& $\bl$ & $\bl$ &$\bl$  & $\bl$ &  $\hl{\bl}$ &  $ $ & $\bl$ &  $ $ &  $ $ &  $ $ &  $ $ & $\bl$ &  $ $ &  $ $ &   $ $&   $ $&$\bl$
\end{tabular} 
\end{minipage}}
\bigskip

\caption{The bijection $\sigma$ shown on the partitions of $\l \in \mup{T_{13}}$, $\l \ne \widetilde{\pi_{13}}$.}
\label{tab_ex22symcol}
\end{table}
\end{example}

\begin{remark}
Another combinatorial equality can be derived from the provided construction for $\mup{N}$. Indeed,
assuming $n=2k-1$ for $k \geq 7$, $h\geq 6$, 
and reasoning as in Example~\ref{exa49}, 
it can be easily  shown that $\#\left(\mup{T_n} \cap \mD_h\right)$ 
 equals the number of partitions in $h-3$ parts of $f(n,h)$ in which each part is smaller than or equal to $g(n,h)$, where
\[
f(n,h)=\frac{-h^3+6h^2+(n-2)h-4n-22}{2}
\text{ \,\,\,\,and\,\,\,\, }
g(n,h)=\frac{n-h^2+3h+1}{2} .
\]
The proof is obtained from Proposition~\ref{prophgen}, considering the bijection 
\begin{equation}\label{lastbije}
a_i \leftrightarrow a_i-a_1+1.
\end{equation}
In Tab.~\ref{tab:fg} the result is summarized for each class.
Notice that, using the bijection of Eq.~\eqref{lastbije} on the partitions shown in Example~\ref{exa49}, one can recover the eleven partitions of 31 in 4 parts,
where each part is not larger than 11.
\begin{table}[]\label{tab:fg}
\begin{tabular}{|c|c|c|}
\hline
class & $f(n,h)$ & $g(n,h)$\\\hline\hline
$\mA_h$ &$\frac{-h^3+6h^2+(n-8)h-4n+2}{2} $ &$\frac{n-h^2+3h-5}{2}$ \\
$\mB_h$ &$\frac{-h^3+6h^2+(n-6)h-4n-6}{2} $ &$\frac{n-h^2+3h-3}{2}$ \\
$\mC_h$ & $\frac{-h^3+6h^2+(n-4)h-4n-14}{2} $ &$\frac{n-h^2+3h-1}{2}$ \\
$\mD_h$ & $\frac{-h^3+6h^2+(n-2)h-4n-22}{2} $& $\frac{n-h^2+3h+1}{2}$\\\hline
\end{tabular}
\bigskip

\caption{The values of $f(n,h)$ and $g(n,h)$ for each class.}
\end{table}
\end{remark}

\section*{Acknowledgments}
The authors are thankful to the referees for their feedback and for their valuable comments improving the quality of the manuscript, and to Alessandro Gambini for useful discussions which motivated the research.
\bibliographystyle{amsalpha}
\bibliography{sym2n_ref.bib}

\providecommand{\bysame}{\leavevmode\hbox to3em{\hrulefill}\thinspace}
\providecommand{\MR}{\relax\ifhmode\unskip\space\fi MR }
\providecommand{\MRhref}[2]{%
  \href{http://www.ams.org/mathscinet-getitem?mr=#1}{#2}
}
\providecommand{\href}[2]{#2}
\begin{thebibliography}{ACGS21b}

\bibitem[ACGS19]{Aragona2019}
R.~Aragona, R.~Civino, N.~Gavioli, and C.~M. Scoppola, \emph{Regular subgroups
  with large intersection}, Ann. Mat. Pura Appl. (4) \textbf{198} (2019),
  no.~6, 2043--2057.

\bibitem[ACGS21a]{aragona2021rigid}
\bysame, \emph{Rigid commutators and a normalizer chain}, Monatsh. Math.
  \textbf{196} (2021), no.~3, 431--455.

\bibitem[ACGS21b]{aragona2021unrefinable}
\bysame, \emph{Unrefinable partitions into distinct parts in a normalizer
  chain}, arXiv preprint arXiv:2107.04666 (2021).

\bibitem[AN19]{andrews2019}
G.~E. Andrews and D.~Newman, \emph{Partitions and the minimal excludant}, Ann.
  Comb. \textbf{23} (2019), no.~2, 249--254.

\bibitem[And76]{andrews1998theory}
G.~E. Andrews, \emph{The theory of partitions}, Encyclopedia of Mathematics and
  its Applications, Vol. 2, Addison-Wesley Publishing Co., Reading,
  Mass.-London-Amsterdam, 1976.

\bibitem[FP15]{fraenkel2015harnessing}
A.~S. Fraenkel and U.~Peled, \emph{Harnessing the unwieldy {MEX} function},
  Games of no chance 4, Math. Sci. Res. Inst. Publ., vol.~63, Cambridge Univ.
  Press, New York, 2015, pp.~77--94.

\bibitem[HSS22]{Hopkins2022}
B.~Hopkins, J.~A. Sellers, and D.~Stanton, \emph{Dyson's crank and the mex of
  integer partitions}, J. Combin. Theory Ser. A \textbf{185} (2022), Paper No.
  105523, 10.

\bibitem[OEI]{OEIS}
\emph{{The On-Line Encyclopedia of Integer Sequences}}, Published
  electronically at \url{https://oeis.org}, Accessed: 2021-11-01.

\end{thebibliography}

\end{document}